\documentclass[reqno]{amsart}
\usepackage{amsmath, amsthm, amssymb, amstext}

\usepackage{hyperref,xcolor}% http://ctan.org/pkg/{hyperref,xcolor}
\hypersetup{
 pdfborder={0 0 0},
 colorlinks,
}
\usepackage{enumitem}
\newtheorem{theorem}{Theorem}
\newtheorem{remark}[theorem]{Remark}

\newtheorem{proposition}[theorem]{Proposition}

\DeclareMathOperator*{\divergenz}{div}              %
\DeclareMathOperator*{\esssup}{ess ~sup}         %

\newcommand{\N}{\mathbb{N}}

\newcommand{\R}{\mathbb{R}}

\newcommand{\RN}{\mathbb{R}^N}

\newcommand{\Lp}[1]{L^{#1}(\Omega)}

\newcommand{\Lprand}[1]{L^{#1}(\partial\Omega)}
\newcommand{\Wp}[1]{W^{1,#1}(\Omega)}

\newcommand{\eps}{\varepsilon}
\newcommand{\ph}{\varphi}

\newcommand{\rand}{\partial\Omega}

\newcommand{\into}{\int_{\Omega}}

\newcommand{\Linf}{L^{\infty}(\Omega)}
\newcommand{\close}{\overline{\Omega}}

\renewcommand{\l}{\left}
\renewcommand{\r}{\right}
\numberwithin{theorem}{section}
\numberwithin{equation}{section}

\title[Global a priori bounds for weak solutions of elliptic systems]{Global a priori bounds for weak solutions of quasilinear elliptic systems with nonlinear boundary condition}

\author[G.\,Marino]{Greta Marino}
\address[G.\,Marino]{Technische Universit\"{a}t Chemnitz, Fakult\"at f\"ur Mathematik, Reichenhainer Stra\ss e 41, 09126 Chemnitz, Germany}
\email{greta.marino@mathematik.tu-chemnitz.de}

\author[P.\,Winkert]{Patrick Winkert}
\address[P.\,Winkert]{Technische Universit\"{a}t Berlin, Institut f\"{u}r Mathematik, Stra\ss e des 17.\,Juni 136, 10623 Berlin, Germany}
\email{winkert@math.tu-berlin.de}



\begin{document}

\begin{abstract}
    In this paper we study quasilinear elliptic systems with nonlinear boundary condition with fully coupled perturbations even on the boundary. Under very general assumptions our main result says that each weak solution of such systems belongs to $L^\infty(\close)\times L^\infty(\close)$. The proof is based on Moser's iteration scheme. The results presented here can also be applied to elliptic systems with homogeneous Dirichlet boundary condition. 
\end{abstract}

\maketitle

%*************************************************
%*************************************************
\section{Introduction}
%*************************************************
%*************************************************

In this paper we study the boundedness of weak solutions of the following quasilinear elliptic system
\begin{equation}\label{problem}
    \begin{aligned}
	-\divergenz \mathcal A_1(x, u, \nabla u)
	&= \mathcal B_1(x, u,v, \nabla u, \nabla v) \qquad 
	&& \text{in } \Omega, \\
	-\divergenz \mathcal A_2(x, v, \nabla v)
	&= \mathcal B_2(x, u, v, \nabla u, \nabla v) \qquad 
	&& \text{in } \Omega, \\
	\mathcal A_1(x, u, \nabla u) \cdot \nu
	&= \mathcal C_1 (x, u, v) 
	&& \text{on } \rand,\\
	\mathcal A_2(x, v, \nabla v) \cdot \nu
	&= \mathcal C_2 (x, u, v) 
	&& \text{on } \rand,
    \end{aligned}
\end{equation}
where $\Omega \subset \R^N$ with $N>1$ is a bounded domain with Lipschitz boundary $\partial \Omega$, $ \nu(x) $ denotes the outer unit normal of $ \Omega$ at $ x \in \partial \Omega$ and the functions $ \mathcal A_i\colon \Omega \times \R   \times \R^N  \to \R^N$, $\mathcal B_i\colon \Omega \times \R \times \R \times \R^N \times \R^N \to \R$, and $ \mathcal{C}_i\colon \rand \times \R \times \R \to \R, i= 1,2 $,  satisfy suitable $(p,q)$-structure conditions with $1<p,q<\infty$.

The main goal of this paper is to prove the existence of a priori bounds for weak solutions of problem \eqref{problem} under very general conditions on the data. Indeed, the novelties of our work can be stated as follows:
\begin{enumerate}
        \item[(i)]
                Problem \eqref{problem} is fully coupled even with the gradient of the solutions and with a  nonlinear boundary condition.
        \item[(ii)]
                Critical growth is allowed even on the boundary.
\end{enumerate}
The proof of our result uses a modified version of Moser's iteration technique whose arguments are essentially based on the monographs of Dr{\'a}bek-Kufner-Nicolosi \cite{Drabek-Kufner-Nicolosi-1997} and Struwe \cite{Struwe-2008}. We extend with our work recent results of the authors \cite{Marino-Winkert-2019} from the case of a single equation to a system which is a difficult task to undertake. To the best of our knowledge, a priori bounds for problem \eqref{problem} under such weak conditions have not been published before and so our results extend several works in this direction.

Let us comment on some relevant references concerning a priori bounds for elliptic systems. In 1992, Cl\'{e}ment-de Figueiredo-Mitidieri \cite{Clement-de-Figueiredo-Mitidieri-1992} studied the semilinear elliptic system
\begin{align}\label{prob_1}
        \begin{cases}
                -\Delta u = f(v) \quad&\text{in }\Omega,\qquad  u=0 \quad\text{on }\partial \Omega,\\
                -\Delta v = g(u) \quad&\text{in }\Omega,\qquad  v=0 \quad\text{on }\partial \Omega,
        \end{cases}
\end{align}
where $f,g$ are smooth functions such that $\alpha,\beta \in (0,\infty)$ exist with
\begin{align*}
        \lim_{s\to 0} \frac{f(s)}{s^p}=\alpha\quad\text{and}\quad \lim_{s\to \infty} \frac{g(s)}{s^q}=\beta,
\end{align*}
where $1\leq p,q<\infty$ satisfy
\begin{align}\label{cond_1}
        \frac{1}{p+1}+\frac{1}{q+1} > \frac{N-2}{N} \quad\text{if }N \geq 3.
\end{align}
Condition \eqref{cond_1} is the crucial assumption in their proof of a priori bounds for weak solutions of \eqref{prob_1} and it can be shown that this condition is optimal. The proof uses the methods applied in the paper of de Figueiredo-Lions-Nussbaum \cite{de-Figueiredo-Lions-Nussbaum-1982} in which condition \eqref{cond_1} first appeared. Since both papers deal not only with a priori bounds but also with the existence of positive solutions, it is worth mentioning the pioneer work of Lions in \cite{Lions-1982} concerning the existence of positive solutions for semilinear elliptic equations. An extension of \cite{Clement-de-Figueiredo-Mitidieri-1992} was done by the same authors in \cite{Clement-de-Figueiredo-Mitidieri-1996} to problems of the form
\begin{align}\label{prob_2}
        \begin{cases}
                -\Delta u = f(x,u,v,Du,Dv) \quad&\text{in }\Omega,\qquad u=0 \quad\text{on }\partial \Omega,\\
                -\Delta v = g(x,u,v,Du,Dv) \quad&\text{in }\Omega,\qquad v=0 \quad\text{on }\partial \Omega,
        \end{cases}
\end{align}
where a priori $L^\infty$-estimates are established for positive solutions of \eqref{prob_2} via a method which combines Hardy-Sobolev-type inequalities and interpolation. In de Figueiredo-Yang \cite{de-Figueiredo-Yang-2001} a priori bounds for solutions of \eqref{prob_2} (without the gradient dependence on $f$ and $g$) are obtained via the so-called blow up method and the results are much more general than those in \cite{Clement-de-Figueiredo-Mitidieri-1996}.

In 2004, a new method for a priori estimates for solutions of semilinear elliptic systems of the form
\begin{align*}%\label{prob_3}
        \begin{cases}
                -\Delta u = f(x,u,v) \quad&\text{in }\Omega,\qquad u=0 \quad\text{on }\partial \Omega,\\
                -\Delta v = g(x,u,v) \quad&\text{in }\Omega,\qquad v=0 \quad\text{on }\partial \Omega,
        \end{cases}
\end{align*}
was presented by Quittner-Souplet \cite{Quittner-Souplet-2004} which is based on a bootstrap argument. In addition, we refer to this work because it gives an overview about the different techniques concerning a priori estimates, see the Introduction of \cite{Quittner-Souplet-2004} and also the references. Concerning a priori estimates for very weak solutions with power nonlinearities we mention the work of Quittner \cite{Quittner-2014}.

A priori bounds and existence of positive solutions for strongly coupled $p$-Laplace systems have been established by Zou \cite{Zou-2006a} for systems given by
\begin{align*}%\label{prob_4}
        \begin{cases}
                -\Delta_m u +u^av^b=0 \quad&\text{in }\Omega,\qquad u=0 \quad\text{on }\partial \Omega,\\
                -\Delta_m v +u^cv^d=0 \quad&\text{in }\Omega,\qquad v=0 \quad\text{on }\partial \Omega,
        \end{cases}
\end{align*}
where $\Delta_m u=\divergenz(|\nabla u|^{m-2}\nabla u)$ denotes the $m$-Laplacian.

In 2010, Bartsch-Dancer-Wang \cite{Bartsch-Dancer-Wang-2010} studied the local and global bifurcation structure of positive solutions of the system
\begin{align}\label{prob_5}
        \begin{cases}
                -\Delta u+u = \mu_1 u^3+\beta v^2u \quad&\text{in }\Omega,\qquad u=0 \quad\text{on }\partial \Omega,\\
                -\Delta v+v = \mu_2 v^3+\beta u^2v \quad&\text{in }\Omega,\qquad v=0 \quad\text{on }\partial \Omega,
        \end{cases}
\end{align}
of nonlinear Schr\"odinger  type equations. They developed a new Liouville type theorem for nonlinear elliptic systems which provides a priori bounds for solution branches of \eqref{prob_5}. Singular quasilinear elliptic systems in $\R^N$ have been recently studied by Marano-Marino-Moussaoui \cite{Marano-Marino-Moussaoui-2019} for $(p_1,p_2)$-Laplace systems given by
\begin{align}\label{prob_6}
        \begin{cases}
                -\Delta_{p_1}u= a_1(x) f(u,v) \quad&\text{in }\R^N,\\
                -\Delta_{p_2}v= a_2(x) g(u,v)\quad&\text{in }\R^N,\\
                u,v>0 & \text{in }\R^N,
        \end{cases}
\end{align}
where a version of Moser's iterations is applied in order to obtain $L^\infty$-bounds for solutions of \eqref{prob_6}, see also Marino \cite{Marino-2019}.

Finally, we refer to other works which are related to a priori bounds and existence of weak solutions of elliptic systems of type \eqref{problem}, see, for example, Angenent-Van der Vorst \cite{Angenent-Van-der-Vorst-2000}, Bahri-Lions \cite{Bahri-Lions-1992}, Choi \cite{Choi-2015}, Damascelli-Pardo \cite{Damascelli-Pardo-2018}, D'Ambrosio-Mitidieri \cite{DAmbrosio-Mitidieri-2018}, Ghergu-R\u{a}dulescu \cite{Ghergu-Radulescu-2008}, Hai \cite{Hai-2011}, Kelemen-Quittner \cite{Kelemen-Quittner-2010}, Kos\'{\i}rov\'{a}-Quittner \cite{Kosirova-Quittner-2011}, Mavinga-Pardo \cite{Mavinga-Pardo-2017}, Mingione \cite{Mingione-2003}, Mitidieri \cite{Mitidieri-1993}, Motreanu \cite{Motreanu-2018}, Motreanu-Moussaoui \cite{Motreanu-Moussaoui-2014a}, \cite{Motreanu-Moussaoui-2014b}, Papageorgiou-R\u{a}dulescu-Repov\v{s} \cite{Papageorgiou-Radulescu-Repovs-2019}, Peletier-Van der Vorst \cite{Peletier-Van-der-Vorst-1992}, Ramos \cite{Ramos-2009}, Souto \cite{Souto-1995}, Troy \cite{Troy-1981}, Zhang \cite{Zhang-1999}, Zhou-Zhang-Liu \cite{Zhou-Zhang-Liu-2012},    Zou \cite{Zou-2006b} and the references therein.

The paper is organized as follows. In Section \ref{sect2} we state the main preliminaries which will be used in the paper.  Section \ref{sect3} contains the main results of our work. First, we prove that any weak solution of \eqref{problem} belongs to $ L^r(\close) \times L^r(\close)$ for any finite $r$, see Theorem \ref{th1} and then, in the second part, we are able to show that each weak solution of \eqref{problem} is essentially bounded, that is, it belongs to $ L^{\infty}(\close) \times L^{\infty}(\close)$, see Theorem \ref{th2}. Furthermore, we will mention that our results can also be applied to problems with homogeneous Dirichlet condition, see Theorem \ref{th3}.

%*************************************************
%*************************************************
\section{Preliminaries}\label{sect2}
%*************************************************
%*************************************************

Throughout the paper we denote by $ \vert \cdot \vert $ the norm of $ \RN$ and $ \cdot $ stands for the inner product in $ \RN$. 
For $r \in [1,\infty)$ we denote  by $ L^r(\Omega), L^r(\Omega; \RN)$ and $ W^{1,r}(\Omega)$ the usual Lebesgue and Sobolev spaces endowed with the norms $ \Vert \cdot \Vert_r$ and $ \Vert \cdot \Vert_{1,r}$  given by
\begin{align*}
    \begin{split}
        \Vert u \Vert_r &= \l(\into \vert u \vert^r \,dx \r)^{\frac{1}{r}}, \quad \Vert \nabla u \Vert_r= \l(\into \vert \nabla u \vert^r \,dx \r)^{\frac{1}{r}}, \\
        \Vert u \Vert_{1,r}&= \l(\into \vert \nabla u \vert^r \,dx \r)^{\frac{1}{r}}+\l(\into \vert u \vert^r \,dx \r)^{\frac{1}{r}}.
    \end{split}
\end{align*}
For $ r= \infty$, the norm of $L^{\infty}(\Omega) $ is given by
\begin{align*}
    \Vert u \Vert_{\infty}= \esssup_{\Omega} \vert u \vert.
\end{align*}
By $\sigma$ we denote the $(N-1)$-dimensional Hausdorff (surface) measure and $ L^s(\rand)$, $ 1 \le s \le \infty$, stands for the Lebesgue space on the boundary with the norms 
\begin{align*}
        \Vert u \Vert_{s, \rand}= \l(\int_{\rand} \vert u \vert^s \,d\sigma \r)^{\frac{1}{s}} \quad (1 \le s< \infty), \qquad \Vert u \Vert_{\infty, \rand}= \esssup_{\rand} \vert u \vert.
\end{align*}
It is well known that the linear trace mapping $ \gamma\colon W^{1,r}(\Omega) \to L^{r_2}(\rand)$ is compact for every $r_2\in [1,r_*)$ and continuous for $r_2=r_*$, where $ r_*$ is the critical exponent of $r$ on the boundary given by
\begin{align}\label{critical1}
        r_*=
        \begin{cases}
            \frac{(N-1)r}{N-r} \qquad & \text{if } r< N, \\
            \text{any } m \in (1, \infty) & \text{if } r \ge N.
        \end{cases}
\end{align}
For simplification we will drop the usage of $\gamma$. Moreover, by the Sobolev embedding theorem, we know that there exists a linear map $ i\colon W^{1,r}(\Omega) \to L^{r_1}(\Omega)$ which is compact for every $r_1\in [1,r^*)$ and continuous for  $r_1=r^*$ where the critical exponent is given by
\begin{align}\label{critical}
        r^*=
        \begin{cases}
                \frac{Nr}{N-r} \qquad & \text{if } r< N, \\
            \text{any } m \in (1, \infty) & \text{if } r \ge N.
        \end{cases}
\end{align}
For $ a \in \R$, we set $ a^{\pm}:= \max\{\pm a, 0\}$ and for $ u \in W^{1,r}(\Omega) $ we define $ u^{\pm}(\cdot):= u(\cdot)^{\pm}$. It is clear  that
\begin{align}\label{pm}
        u^{\pm} \in W^{1,r}(\Omega), \quad \vert u \vert= u^+ + u^-, \quad u= u^+- u^-.
\end{align}
Moreover, $ \vert \cdot \vert $ stands for the Lebesgue measure on $ \RN$ and also for the Hausdorff surface measure and it will be clear from the context which one is used. If $s>1$, then $s':=\frac{s}{s-1}$ denotes its conjugate. 

The following propositions are needed in the proofs of our main results.

\begin{proposition}(\cite[Proposition 2.1]{Winkert-2014})\label{proposition_boundary_integral}
    Let $\Omega \subset \R^N$, $N>1,$ be a bounded domain with Lipschitz boundary $\rand$, let $1<p<\infty$, and let $\hat{q}$ be such that $p \leq \hat{q}<p_*$ with the critical exponent stated in \eqref{critical1} with $r=p$.
    Then, for every $\eps>0$, there exist constants $\tilde{c}_1>0$ and $\tilde{c}_2>0$ such that
    \begin{align*}
	\|u\|_{\hat{q},\rand}^p \leq \eps \|u\|_{1,p}^p+\tilde{c}_1\eps^{-\tilde{c}_2} \|u\|_{p}^p \qquad \text{for all }u \in W^{1,p}(\Omega).
    \end{align*}
\end{proposition}

\begin{proposition}(\cite[Proposition 2.2]{Marino-Winkert-2019})\label{proposition_final_interation}
    Let $\Omega \subset \R^N$, $N>1,$ be a bounded domain with Lipschitz boundary $\rand$. Let $u \in \Lp{p}$ with $u\geq 0$ and $1<p<\infty$ such that
    \begin{align*}%\label{estimate_final}
	\|u\|_{\alpha_n}\leq C
    \end{align*}
    with a constant $C>0$ and a sequence $(\alpha_n)\subseteq \R_+$ with $\alpha_n\to \infty$ as $n\to \infty$. Then, $u \in \Linf$.
\end{proposition}

\begin{proposition}(\cite[Proposition 2.4]{Marino-Winkert-2019})\label{proposition_boundedness_boundary}
    Let $\Omega \subset \R^N$, $N>1,$ be a bounded domain with Lipschitz boundary $\rand$ and let $1<p<\infty$. %If $u \in \Wp{p}\cap \Lp{r}$ for any $r\in (1,\infty)$, then $u \in \Lprand{r}$ for any $r\in (1,\infty)$. In particular, 
    If $u \in \Wp{p}\cap \Linf$, then $u \in L^\infty(\partial\Omega)$.
\end{proposition}

In the following we will use the abbreviation
\begin{align*}
        L^\infty(\close):=L^\infty(\Omega) \cap L^\infty(\partial\Omega).
\end{align*}

% \begin{lemma}\label{conjugate}
%         Let $ p, q> 1 $ be such that $p> q$ and let $ p', q'$ be their conjugate exponents. Then, $ p'< q'$.
% \end{lemma}
% 
% \begin{proof}
%         We only need to observe that, since $ -\frac{1}{p}> -\frac{1}{q}$, then we have 
%         \begin{align*}
%             \frac{1}{p'}= 1- \frac{1}{p}> 1- \frac{1}{q}= \frac{1}{q'},
%         \end{align*}
%         that is $ p'< q'$. This completes the proof.
% \end{proof}

%*************************************************
%*************************************************
\section{Main results}\label{sect3}
%*************************************************
%*************************************************

We now give the structure conditions on the nonlinearities in problem \eqref{problem}.

\begin{enumerate}[leftmargin=0.7cm]
 \item[(H)] 
    The functions $ \mathcal A_i\colon \Omega \times \R   \times \R^N  \to \R^N$, $\mathcal B_i\colon \Omega \times \R \times \R \times \R^N \times \R^N \to \R$ and $ \mathcal C_i\colon \rand \times \R \times \R \to \R, i= 1,2 $, are Carath\'eodory functions such that the following holds:
    \begin{align*}
        \text{(H1}) \quad  
	    & \vert \mathcal A_1(x, s, \xi) \vert \le A_1 \vert \xi \vert^{p-1}+  A_2 \vert s \vert^{r_1 \frac{p-1}{p}}+ A_3,  \\
        \text{(H2}) \quad 
	    & \vert \mathcal A_2(x, t, \zeta) \vert \le  \tilde A_1 \vert \zeta \vert^{q-1}+ \tilde A_2  \vert t \vert^{r_2 \frac{q-1}{q}}+ \tilde A_3,  \\
        \text{(H3}) \quad 
	    & \mathcal A_1(x, s, \xi) \cdot \xi \ge A_4 \vert \xi \vert^p- A_5 \vert s \vert^{r_1}- A_6,   \\
        \text{(H4}) \quad 
	    & \mathcal A_2(x, t, \zeta) \cdot \zeta \ge  \tilde A_4 \vert \zeta \vert^q-  \tilde A_5 \vert t \vert^{r_2}- \tilde A_6,  \\
        \text{(H5}) \quad 
	    & \vert \mathcal B_1(x, s, t, \xi, \zeta) \vert \le B_1 \vert s \vert^{b_1}+ B_2 \vert t \vert^{b_2}+ B_3 \vert s \vert^{b_3} \vert t \vert^{b_4}+ B_4 \vert \xi \vert^{b_5}\\
	    & \hspace*{3cm}+ B_5 \vert \zeta \vert^{b_6}+ B_6 \vert \xi \vert^{b_7} \vert \zeta \vert^{b_8}+ B_7,  \\
        \text{(H6}) \quad 
	    & \vert \mathcal B_2(x, s, t, \xi, \zeta) \vert \le  \tilde B_1 \vert s \vert^{\tilde b_1}+ \tilde B_2 \vert t \vert^{\tilde b_2}+ \tilde B_3  \vert s \vert^{\tilde b_3} \vert t \vert^{\tilde b_4}+ \tilde B_4  \vert \xi \vert^{\tilde b_5}\\
	    & \hspace*{3cm}
	    + \tilde B_5  \vert \zeta \vert^{\tilde b_6}+ \tilde B_6  \vert \xi \vert^{\tilde b_7} \vert \zeta \vert^{\tilde b_8}+ \tilde B_7, \\
        \text{(H7}) \quad 
	    & \vert \mathcal C_1(x, s,t) \vert \le C_1 \vert s \vert^{c_1}+ C_2 \vert t \vert^{c_2}+ C_3 \vert s \vert^{c_3} \vert t \vert^{c_4}+ C_4,  \\
        \text{(H8}) \quad 
	    & \vert \mathcal C_2(x, s,t) \vert \le  \tilde C_1 \vert s \vert^{\tilde c_1}+ \tilde C_2  \vert t \vert^{\tilde c_2}+ \tilde C_3 \vert s \vert^{\tilde c_3} \vert t \vert^{\tilde c_4}+ \tilde C_4, 
\end{align*}
for a.e.\,$x \in \Omega$, respectively for a.e.\,$x \in \rand$, for all $ s, t \in \R $, for all $ \xi, \zeta \in \R^N $, with nonnegative constants $ A_i, \tilde A_i, B_j, \tilde B_j, C_k, \tilde C_k \,  (i \in \{1, \dots, 6\}$, $j \in \{1, \dots, 7\}$, $k \in \{1,\dots, 4\}) $ and with $ 1< p, q< \infty$. Moreover, the exponents $ b_i, \tilde b_i, c_j, \tilde c_j, r_1, r_2  \, (i \in \{1, \dots, 8\}, j \in \{1, \dots, 4\})$ are nonnegative and satisfy the following assumptions
\begin{align*}
      \hspace*{0.3cm}\begin{array}[]{lll}
		        \text{(E1)} \quad  r_1 \le p^*
                & \text{(E2)} \quad  r_2 \le q^*
                &\\[3ex]
		\text{(E3)} \quad b_1 \le p^*-1
                & \text{(E4)} \quad \displaystyle b_2 < \frac{q^*}{p^*}(p^*-p)
                & \text{(E5)} \quad  \displaystyle\frac{b_3}{p^*}+ \frac{b_4}{q^*}< \frac{p^*-p}{p^*}\\[3ex]
                \text{(E6)} \quad b_5 \le p-1
                & \text{(E7)} \quad \displaystyle b_6 < \frac{q}{p^*}(p^*-p)
                & \text{(E8)} \quad\displaystyle \frac{b_7}{p}+ \frac{b_8}{q}< \frac{p^*-p}{p^*}\\[3ex]
                \text{(E9)}\displaystyle \quad \tilde b_1 < \frac{p^*}{q^*}(q^*-q)
                & \text{(E10)} \quad \displaystyle \tilde b_2 \le q^*-1
                & \text{(E11)} \quad\displaystyle\frac{\tilde b_3}{p^*}+ \frac{\tilde b_4}{q^*}< \frac{q^*-q}{q^*}\\[3ex]
                \text{(E12)} \displaystyle\quad  \tilde b_5 < \frac{p}{q^*}(q^*-q)
                & \text{(E13)} \quad \displaystyle\tilde b_6 \le q- 1
                & \text{(E14)} \quad\displaystyle\frac{\tilde b_7}{p}+ \frac{\tilde b_8}{q}< \frac{q^*-q}{q^*}\\[3ex]
                 \text{(E15)} \quad c_1  \le p_*-1
                & \text{(E16)} \quad \displaystyle c_2 < \frac{q_*}{p_*}(p_*-p)
                & \text{(E17)} \quad\displaystyle\frac{c_3}{p_*}+ \frac{c_4}{q_*}< \frac{p_*-p}{p_*}\\[3ex]
                \text{(E18)}\displaystyle \quad  \tilde c_1 < \frac{p_*}{q_*}(q_*-q)
                & \text{(E19)} \quad \displaystyle\tilde c_6 \le q_*-1
                & \text{(E20)} \quad\displaystyle\frac{\tilde c_3}{p_*}+ \frac{\tilde c_4}{q_*}< \frac{q_*-q}{q_*},
        \end{array}
\end{align*}
where the numbers $ p^*, p_*, q^*, q_*$ are defined by \eqref{critical} and \eqref{critical1}. 
\end{enumerate}

A couple $ (u,v) \in W^{1,p}(\Omega) \times W^{1,q}(\Omega) $ is said to be a weak solution of problem \eqref{problem} if
        \begin{align}\label{weak}
            \begin{split}
                    \into \mathcal A_1(x, u, \nabla u) \cdot \nabla \varphi \,dx&= \into \mathcal B_1(x, u, v, \nabla u, \nabla v) \varphi \,dx+ \int_{\rand} \mathcal C_1(x, u, v) \varphi \,d\sigma \\
                    \into \mathcal A_2(x, v, \nabla v) \cdot \nabla \psi \,dx&= \into \mathcal B_2(x, u, v, \nabla u, \nabla v) \psi \,dx+ \int_{\rand} \mathcal C_2(x, u, v) \psi \,d\sigma \\
            \end{split}
        \end{align}
holds for all $ (\varphi, \psi) \in W^{1,p}(\Omega) \times W^{1,q}(\Omega) $. By hypotheses (H) and the Sobolev embedding along with the continuity of the trace operator it is clear that this definition of a weak solution is well-defined. Indeed, if we estimate the integral concerning the function $\mathcal B_1\colon \Omega \times \R \times \R \times \R^N \times \R^N \to \R$  using condition (H5) we obtain several mixed terms. Let us consider, for example, the third term on the right-hand side of (H5). Applying H\"older's inequality we get
\begin{align}\label{well-definedness}
    \begin{split}
        &B_3\into |u|^{b_3}|v|^{b_4} \ph \,dx\\
        &\leq B_3\left(\into|u|^{b_3 s_1}\,dx\right)^{\frac{1}{s_1}}\left(\into|v|^{b_4 s_2}\,dx\right)^{\frac{1}{s_2}} \left(\into|\ph|^{ s_3}\,dx\right)^{\frac{1}{s_3}},
    \end{split}
\end{align}
where $ (u,v) \in W^{1,p}(\Omega) \times W^{1,q}(\Omega) $, $\ph \in \Wp{p}$ and
\begin{align*}
        \frac{1}{s_1}+\frac{1}{s_2}+\frac{1}{s_3}=1.
\end{align*}
Taking $s_3=p^*$ and using $s_1 \leq \frac{p^*}{b_3}$ as well as $s_2 \leq \frac{q^*}{b_4}$  leads to 
\begin{align}\label{condition_weak}
        \frac{b_3}{p^*}+ \frac{b_4}{q^*}\leq \frac{p^*-1}{p^*}.
\end{align}
This condition is necessary for the finiteness of the integrals of the right-hand side of \eqref{well-definedness}, see also Remark \ref{remark_well_definitedness}. Since we need some stronger conditions in order to apply Moser's iteration, we suppose condition (E5) which implies \eqref{condition_weak}. In the same way we can prove the finiteness of all integrals in the definition of \eqref{weak}.

Our first result shows that any weak solution of problem \eqref{problem} belongs to the space $ L^r(\close) \times L^r(\close)$ for any finite $r$.

\begin{theorem} \label{th1}
    Let $ \Omega \subset \R^N, N> 1$, be a bounded domain with Lipschitz boundary $ \rand $ and let hypotheses (H) be satisfied. Then, every weak solution $ (u, v) \in W^{1,p}(\Omega) \times W^{1,q}(\Omega) $ of problem \eqref{problem} belongs to $ L^{r}(\close) \times L^{r}(\close)$ for every $ r \in (1, \infty)$. 
\end{theorem}

\begin{proof}
    Let $ (u, v) \in W^{1,p}(\Omega) \times W^{1,q}(\Omega) $ be a weak solution of \eqref{problem} in the sense of \eqref{weak}. We only show that $ u \in L^r(\close)$, the proof for $v$ can be done in the same way. Moreover, taking \eqref{pm} into account, without any loss of generality, we can assume that $ u, v \ge 0 $ (otherwise we prove the result for $u^+, v^+$ and $u^-, v^-$, respectively). Moreover, throughout the proof we will denote by $M_i$, $i= 1, 2, \dots$, constants which may depend on some natural norms of $u$ and $v$.%, and we will write explicitly this dependance  when it will be the case. 
    
    For every $ h \ge 0 $ we set $ u_h:= \min\{u, h\}$ and choose $ \varphi= u u_h^{\kappa p} \in \Wp{p}$ for $\kappa>0$ as test function in the first equation of \eqref{weak}. Since $ \nabla \varphi= u_h^{\kappa p} \nabla u + \kappa p u  u_h^{\kappa p- 1} \nabla u_h$ this results in
    \begin{align}\label{1}
        \begin{split}
            & \into (\mathcal A_1(x, u, \nabla u) \cdot \nabla u) u_h^{\kappa p} \,dx+ \kappa p \into (\mathcal A_1(x, u, \nabla u) \cdot \nabla u_h) u u_h^{\kappa p-1}  \,dx \\
            &= \into \mathcal B_1(x, u, v, \nabla u, \nabla v) u u_h^{\kappa p} \,dx+ \int_{\rand} \mathcal C_1(x, u, v) u u_h^{\kappa p} \,d\sigma.
        \end{split}
    \end{align}
    Now we apply (H3) to the first term on the left-hand side of \eqref{1} which gives
    \begin{align*}%\label{2}
        \begin{split} 
            &\into (\mathcal A_1(x, u, \nabla u) \cdot \nabla u) u_h^{\kappa p} \,dx \\
            & \ge \into \left(A_4 \vert \nabla u \vert^p- A_5 u^{r_1}- A_6 \r) u_h^{\kappa p} \,dx \\
            & \ge A_4 \into \vert \nabla u \vert^p u_h^{\kappa p} \,dx- (A_5+A_6) \into u^{p^*} u_h^{\kappa p} \,dx- (A_5+A_6) \vert \Omega \vert.
        \end{split}
    \end{align*}
    In the same way we use (H3) to the second term on the left-hand side. This shows
    \begin{align*}%\label{3}
        \begin{split}
            &\kappa p \into (\mathcal A_1(x, u, \nabla u) \cdot \nabla u_h) u u_h^{\kappa p-1}  \,dx \\
            &= \kappa p \int_{\{x \in \Omega: \, u(x) \le h\}} (\mathcal A_1(x, u, \nabla u) \cdot \nabla u) u_h^{\kappa p} \,dx \\
            & \ge \kappa p \int_{\{x \in \Omega: \, u(x) \le h\}} \l(A_4 \vert \nabla u \vert^p- A_5 u^{r_1}- A_6 \r) u_h^{\kappa p} \,dx \\
            & \ge A_4 \kappa p \int_{\{x \in \Omega: \, u(x) \le h\}} \vert \nabla u \vert^p u_h^{\kappa p} \,dx\\
            & \quad - \kappa p (A_5+ A_6)  \into u^{p^*} u_h^{\kappa p} \,dx- \kappa p (A_5+A_6) \vert \Omega \vert.
        \end{split}
    \end{align*}
    Taking (H5) into account we get for the first term on the right-hand side of \eqref{1} the following estimate 
    \begin{align}\label{4}
        \begin{split}
            & \into \mathcal B_1(x, u, v, \nabla u, \nabla v) u u_h^{\kappa p} \,dx \\
            & \le \into \l(B_1 u^{b_1}+ B_2 v^{b_2}+ B_3 u^{b_3} v^{b_4}+ B_4 \vert \nabla u \vert^{b_5}\right.\\
            & \quad \left.+ B_5 \vert \nabla v \vert^{b_6}+ B_6 \vert \nabla u \vert^{b_7} \vert \nabla v \vert^{b_8}+ B_7  \r) u u_h^{\kappa p} \,dx.
        \end{split}
    \end{align}
    We are going to estimate each term of the inequality above separately. First, taking into account assumption (E3), we have
    \begin{align*}
        B_1 \into u^{b_1} u u_h^{\kappa p} \,dx \le B_1 \into u^{p^*} u_h^{\kappa p} \,dx+ B_1 \vert \Omega \vert.
    \end{align*}
    Moreover, thanks to H\"older's inequality with  $ s_1> 1$ such that $ b_2 s_1 = q^*$, which is possible by (E4), we have
    \begin{align*}
	\begin{split}
	    B_2 \into v^{b_2} u u_h^{\kappa p} \,dx 
	    & \le B_2 \l(\into v^{b_2 s_1} \,dx \r)^{1/{s_1}} \l(\into (u u_h^{\kappa p})^{s_1'} \,dx \r)^{1/{s_1'}} \\
	    & \le M_1 \l(1+ \Vert u u_h^{\kappa} \Vert_{p s_1'}^p \r).
	\end{split}
    \end{align*}
    Applying again H\"older's inequality with exponents $x_1, y_1, z_1>1 $ such that
    \begin{equation}\label{cases1}
        b_3 x_1 = p^*, \qquad   b_4 y_1 = q^*, \qquad  \frac{1}{z_1}= 1- \frac{1}{x_1}- \frac{1}{y_1}
    \end{equation}
    leads to
    \begin{align*}%\label{b3}
        \begin{split}
            & B_3 \into u^{b_3} v^{b_4} u u_h^{\kappa p} \,dx\\
            & \le B_3 \l(\into u^{b_3 x_1} \,dx \r)^{1/{x_1}} \l(\into v^{b_4 y_1} \,dx \r)^{1/{y_1}} \l(\into (u u_h^{\kappa p})^{z_1} \,dx \r)^{1/{z_1}} \\
            & \le M_2 \l(1+ \Vert u u_h^{\kappa} \Vert_{pz_1}^p \r).
        \end{split}
    \end{align*}
    Note that from (E5) it follows that $b_3<p^*$ as well as $b_4<q^*$ and so the choice in \eqref{cases1} is possible. Thanks to Young's inequality with $ \frac{p}{b_5}> 1 $ we have
    \begin{align*}
        \begin{split}
            B_4 \into \vert \nabla u \vert^{b_5} u u_h^{\kappa p}\,dx 
            &= B_4 \into \l(\l(\frac{A_4}{2 B_4}\r)^{\frac{b_5}{p}} \vert \nabla u \vert^{b_5} u_h^{\kappa b_5} \r) \l(\l(\frac{A_4}{2 B_4}\r)^{-\frac{b_5}{p}} u u_h^{\kappa (p- b_5)} \r) \,dx \\
            & \le \frac{A_4}{2} \into \vert \nabla u \vert^p u_h^{\kappa p} \,dx+ B_4 \l(\frac{A_4}{2 B_4}\r)^{-\frac{b_5}{p-b_5}} \into u^{\frac{p}{p-b_5}} u_h^{\kappa p} \,dx \\
            & \le \frac{A_4}{2}  \into \vert \nabla u \vert^p u_h^{\kappa p} \,dx+ M_3 \l(1+ \into u^{p^*} u_h^{\kappa p} \,dx\r).
        \end{split}
    \end{align*}
    We apply H\"older's inequality with $ s_2> 1$ such that $ b_6 s_2 = q $ in order to get
    \begin{align*}
        \begin{split}
            B_5 \into \vert \nabla v \vert^{b_6} u u_h^{\kappa p} \,dx 
            & \le B_5 \l(\into \vert \nabla v \vert^{b_6 s_2} \,dx \r)^{1/{s_2}} \l(\into (u u_h^{\kappa p})^{s_2'} \,dx \r)^{1/{s_2'}} \\
            & \le M_4 \l(1+ \Vert u u_h^{\kappa} \Vert_{p s_2'}^p\r).
        \end{split}
    \end{align*}
    As before, by H\"older's inequality with $ x_2, y_2, z_2 >1$ such that
    \begin{equation} \label{cases2}
            b_7 x_2 = p,\qquad   b_8 y_2 = q,\qquad  \frac{1}{z_2}= 1- \frac{1}{x_2}- \frac{1}{y_2}
    \end{equation}
   we obtain
     \begin{align*}
        \begin{split}
            &B_6 \into \vert \nabla u \vert^{b_7} \vert \nabla v \vert^{b_8} u u_h^{\kappa p} \,dx\\ 
            & \le B_6 \l(\into \vert \nabla u \vert^{b_7 x_2} \,dx \r)^{1/{x_2}} \l(\into \vert \nabla v \vert^{b_8 y_2} \,dx \r)^{1/{y_2}} \l(\into (u u_h^{\kappa})^{z_2} \,dx \r)^{1/{z_2}} \\
            & \le M_5 \l(1+ \Vert u u_h^{\kappa} \Vert_{p z_2}^p \r),
        \end{split}
    \end{align*}
    which is possible because of (E8).   Finally, for the last term on the right-hand side of \eqref{4} we have
    \begin{align*}
            B_7 \into u u_h^{\kappa p}\,dx \le B_7 \into u^{p^*} u_h^{\kappa p} \,dx+ B_7 \vert \Omega \vert.
    \end{align*}
    Hypothesis (H7) gives  the following estimate for the boundary term of \eqref{1}
    \begin{align}\label{c1}
            \int_{\rand} \mathcal C_1(x, u, v) u u_h^{\kappa p} \,d\sigma \le \int_{\rand} \l(C_1 u^{c_1}+ C_2 v^{c_2}+ C_3 u^{c_3} v^{c_4}+ C_4 \r) u u_h^{\kappa p} \,d\sigma.
    \end{align}
    Exploiting the condition on $ c_1$ in the first term of \eqref{c1} and applying H\"older's inequality with $ t_1> 1$ such that $ c_2 t_1 = q_*$ to the second one we have
    \begin{align*}
        C_1 \int_{\rand} u^{c_1+1} u_h^{\kappa p} \,d\sigma \le C_1 \int_{\rand} u^{p_*} u_h^{\kappa p} \,d\sigma+ C_1 \vert \rand \vert
    \end{align*}
    and
    \begin{align*}
        \begin{split}
            C_2 \int_{\rand} v^{c_2} u u_h^{\kappa p} \,d\sigma & \le C_2 \l(\int_{\rand} v^{c_2 t_1} \,d\sigma\r)^{1/{t_1}} \l(\int_{\rand} (u u_h^{\kappa p})^{t_1'} \,d\sigma \r)^{1/{t_1'}} \\
            & \le M_6 \l(1+ \Vert u u_h^{\kappa} \Vert_{p t_1', \rand}^p \r),
        \end{split}
    \end{align*}
    respectively. For the third term of \eqref{c1} we  apply H\"older's inequality with exponents $ x_3, y_3, z_3 >1$ such that
    \begin{equation} \label{cases3}
            c_3 x_3 = p_*, \qquad    c_4 y_3 = q_*, \qquad  \frac{1}{z_3}= 1- \frac{1}{x_3}- \frac{1}{y_3}
    \end{equation}
    in order to get
    \begin{align*}
        \begin{split}
            & C_3 \int_{\rand} u^{c_3} v^{c_4} u u_h^{\kappa p} \,d\sigma\\
            & \le C_3 \l(\int_{\rand} u^{c_3 x_3} \,d\sigma \r)^{1/{x_3}} \l(\int_{\rand} v^{c_4 y_3} \,d\sigma \r)^{1/{y_3}} \l(\int_{\rand} (u u_h^{\kappa p})^{z_3} \,d\sigma \r)^{1/{z_3}} \\
            & \le M_7 \l(1+ \Vert u u_h^{\kappa } \Vert_{pz_3, \rand}^p \r).
        \end{split}
    \end{align*}
    Finally, for the last term of \eqref{c1} we have    
    \begin{align*}
            C_4 \int_{\rand} u u_h^{\kappa p} \,d\sigma \le C_4 \int_{\rand} u^{p_*} u_h^{\kappa p} \,d\sigma+ C_4 \vert \rand \vert.
    \end{align*}
    Note that from the choice of $s_1, s_2$ and $t_1$ in combination with (E4), (E7) and (E16) we have
    \begin{align*}
            s_1', s_2'< \frac{p^*}{p} \quad\text{and}\quad t_1' <\frac{p_*}{p}.
    \end{align*}
    Furthermore, by  \eqref{cases1}, \eqref{cases2}, \eqref{cases3} and the conditions (E5), (E8) and (E17) we see that
    \begin{align*}
            z_1, z_2 <\frac{p^*}{p}\quad \text{and} \quad z_3<\frac{p_*}{p}.
    \end{align*}   
    Now we combine all the calculations above and set  
\begin{equation}
\label{s}
s:= \max\{s_1', s_2', z_1, z_2\} \in \left(1,\frac{p^*}{p}\right)
\end{equation}
as well as 
\begin{equation}
\label{t}
t:= \max\{t_1', z_3\}\in \left(1,\frac{p_*}{p}\right)
\end{equation}
which finally gives
    \begin{align*}
        \begin{split}
            & A_4 \biggl(\frac{1}{2} \into \vert \nabla u \vert^p u_h^{\kappa p} \,dx
            + \kappa p \int_{\{x \in \Omega: \, u(x) \le h\}} \vert \nabla u \vert^p u_h^{\kappa p} \,dx \biggr) \\
            & \le \l [(\kappa p+1)(A_5+ A_6)+ B_1+ B_7+ M_3 \r] \into u^{p^*} u_h^{\kappa p} \,dx + (C_1+ C_4) \int_{\rand} u^{p_*} u_h^{\kappa p} \,d\sigma\\
            & \qquad+ M_8  \Vert u u_h^{\kappa} \Vert_{p s}^p 
            + M_{9} \Vert u u_h^{\kappa} \Vert_{pt, \rand}^p+ M_{10} (\kappa+1).
        \end{split}
    \end{align*}
    Simplifying the inequality above leads to
    \begin{align}\label{6}
        \begin{split}
            & \frac{A_4}{2} \frac{\kappa p+1}{(\kappa+1)^p} \into \vert \nabla u u_h^{\kappa} \vert^p \,dx \\
            & \le M_{11}(\kappa p+1) \into u^{p^*} u_h^{\kappa p} \,dx+  M_{12} \int_{\rand} u^{p_*} u_h^{\kappa p} \,d\sigma+ M_8 \Vert u u_h^{\kappa} \Vert_{p s}^p \\
            & \qquad + M_{9} \Vert u u_h^{\kappa} \Vert_{pt, \rand}^p+ M_{10} (\kappa+1),
        \end{split}
    \end{align}
    see Marino-Winkert \cite[Inequality after (3.7)]{Marino-Winkert-2019}.
    Dividing by $\frac{A_4}{2}$, summarizing the constants and adding on both sides of \eqref{6} the nonnegative term $ \frac{\kappa p+1}{(\kappa+1)^p} \Vert u u_h^{\kappa} \Vert_p^p$ gives
    \begin{align}\label{dis}
        \begin{split}
            & \frac{\kappa p+1}{(\kappa+1)^p} \Vert u u_h^{\kappa} \Vert_{1,p}^p \\
            & \le \frac{\kappa p+1}{(\kappa+1)^p} \Vert u u_h^{\kappa} \Vert_p^p+ M_{13} (\kappa p+1) \into u^{p^*} u_h^{\kappa p} \,dx+ M_{14} \int_{\rand} u^{p_*} u_h^{\kappa p} \,d\sigma \\
            & \qquad + M_{15} \Vert u u_h^{\kappa} \Vert_{p s}^p+ M_{16} \Vert u u_h^{\kappa} \Vert_{pt, \rand}^p+ M_{17} (\kappa+1) \\
            & \le M_{18} \l(\frac{\kappa p+1}{(\kappa+1)^p}+ 1\r) \Vert u u_h^{\kappa} \Vert_{ps}^p+ M_{13} (\kappa p+1) \into u^{p^*} u_h^{\kappa p} \,dx \\
            & \qquad + M_{14} \int_{\rand} u^{p_*} u_h^{\kappa p} \,d\sigma+ M_{16} \Vert u u_h^{\kappa} \Vert_{p t, \rand}^p+ M_{17} (\kappa+1),
        \end{split}
    \end{align}
    where we applied H\"older's inequality in the last passage. 
    
    Now, let $L, G> 0 $ and set $ a:= u^{p^*-p}$ and $ b:= u^{p_*-p}$. By using H\"older's inequality and the continuous embeddings $i\colon \Wp{p}\to \Lp{p^*}$ and $\gamma\colon \Wp{p}\to \Lprand{p_*}$ we obtain
    \begin{align}\label{crit}
	\begin{split}
	    & \into u^{p^*} u_h^{\kappa p} \,dx \\
	    & = \int_{\{x \in \Omega: \, a(x) \le L\}} a (u u_h^{\kappa})^p \,dx+ \int_{\{x \in \Omega: \, a(x)> L\}} a (u u_h^{\kappa})^p \,dx \\
	    & \le L \into (u u_h^{\kappa})^p \,dx \\
	    & \qquad + \l(\int_{\{x \in \Omega: \, a(x)> L\}} a^{\frac{p^*}{p^*-p}} \,dx \r)^{\frac{p^*-p}{p}} \l(\into (u u_h^{\kappa})^{p^*} \,dx \r)^{\frac{p}{p^*}} \\
	    & \le L \vert \Omega \vert^{1/{s'}} \Vert u u_h^{\kappa} \Vert_{ps}^p+ \l(\int_{\{x \in \Omega: \, a(x)> L\}} a^{\frac{p^*}{p^*-p}} \,dx \r)^{\frac{p^*-p}{p^*}} c_{\Omega}^p \Vert u u_h^{\kappa} \Vert_{1,p}^p
	\end{split}
    \end{align}
    and
    \begin{align}\label{crit1}
	\begin{split}
	    & \int_{\rand} u^{p_*} u_h^{\kappa p} \,d\sigma \\
	    &= \int_{\{x \in \rand: \, b(x) \le G\}} b (u u_h^{\kappa})^p \,d\sigma+ \int_{\{x \in \rand: \, b(x)> G\}} b (u u_h^{\kappa})^p \,d\sigma \\
	    & \le G \int_{\rand} (u u_h^{\kappa})^p \,d\sigma \\
	    & \qquad + \l(\int_{\{x \in \rand: \, b(x)> G\}} b^{\frac{p_*}{p_*-p}} \,d\sigma \r)^{\frac{p_*-p}{p_*}} \l(\int_{\rand} (u u_h^{\kappa})^{p_*} \,d\sigma \r)^{\frac{p}{p_*}} \\
	    & \le G \vert \rand \vert^{1/{t'}} \Vert u u_h^{\kappa} \Vert_{p t, \rand}+ \l(\int_{\{x \in \rand: \, b(x)> G\}} b^{\frac{p_*}{p_*-p}} \,d\sigma \r)^{\frac{p_*-p}{p_*}} c_{\rand}^p \Vert u u_h^{\kappa} \Vert_{1,p}^p
	\end{split}
    \end{align}
    with the embedding constants $c_{\Omega}$ and $ c_{\rand}$. We point out that
    \begin{equation}\label{limit}
	\begin{aligned}
	    & H(L):= \l(\int_{\{x \in \Omega: \, a(x)>L\}} a^{\frac{p^*}{p^*-p}} \,dx \r)^{\frac{p^*-p}{p^*}} \to 0 \quad && \text{as } L \to \infty, \\
	    & K(G):= \l(\int_{\{x \in \rand: \, b(x)> G\}} b^{\frac{p_*}{p_*-p}} \,d\sigma \r)^{\frac{p_*-p}{p_*}} \to 0 && \text{as } G \to \infty.
	\end{aligned}
    \end{equation}
    Combining \eqref{dis}, \eqref{crit}, \eqref{crit1} and \eqref{limit} yields
    \begin{align}\label{dis1}
	\begin{split}
	    & \frac{\kappa p+1}{(\kappa+1)^p}  \Vert u u_h^{\kappa} \Vert_{1,p}^p \\
	    & \le M_{19} \l(\frac{\kappa p+1}{(\kappa+1)^p}+ 1+ (\kappa p+1) L \vert \Omega \vert^{1/{s'}} \r) \Vert u u_h^{\kappa} \Vert_{ps}^p \\
	    & \qquad + M_{13} (\kappa p+1) H(L) c_{\Omega}^p \Vert u u_h^{\kappa} \Vert_{1,p}^p+ (M_{16}+ M_{14} G \vert \rand \vert^{1/{t'}}) \Vert u u_h^{\kappa} \Vert_{p t, \rand}^p \\
	    & \qquad + M_{14} K(G) c_{\rand}^p \Vert u u_h^{\kappa} \Vert_{1,p}^p+ M_{17} (\kappa+1).
	\end{split}
    \end{align}
    Taking \eqref{limit} into account we choose $ L= L(\kappa, u)> 0 $ and $ G= G(\kappa, u)> 0 $ such that
    \begin{align*}
	M_{13} (\kappa p+1) H(L) c_{\Omega}^p= \frac{\kappa p+1}{4(\kappa+1)^p} \quad \text{and} \quad M_{14} K(G) c_{\rand}^p= \frac{\kappa p+1}{4(\kappa+1)^p}.
    \end{align*}
    Therefore, inequality \eqref{dis1} can be written as
    \begin{align}\label{dis2}
	\begin{split}
	    & \frac{\kappa p+1}{2(\kappa+1)^p} \Vert u u_h^{\kappa} \Vert_{1,p}^p \\
	    & \le M_{19} \l(\frac{\kappa p+1}{(\kappa+1)^p}+ 1+ (\kappa p+1) L(\kappa, u) \vert \Omega \vert^{1/{s'}} \r) \Vert u u_h^{\kappa} \Vert_{ps}^p \\
	    &\qquad + (M_{16}+ M_{14} G(\kappa, u) \vert \rand \vert^{1/{t'}}) \Vert u u_h^{\kappa} \Vert_{p t, \rand}^p+ M_{17} (\kappa+1).
	\end{split}
    \end{align}
Taking into account \eqref{t} we have $ p t< p_*$. Thus, we can apply Proposition \ref{proposition_boundary_integral} to estimate the boundary term in \eqref{dis2}. This gives
    \begin{align}\label{bound}
	\begin{split}
	    \Vert u u_h^{\kappa} \Vert_{p t, \rand}^p & \le \eps_1 \Vert u u_h^{\kappa} \Vert_{1,p}^p+ \tilde c_1 \eps_1^{-\tilde c_2} \Vert u u_h^{\kappa} \Vert_p^p \\
	    & \le  \eps_1 \Vert u u_h^{\kappa} \Vert_{1,p}^p+ \tilde c_1 \eps_1^{-\tilde c_2} \vert \Omega \vert^{1/{s'}} \Vert u u_h^{\kappa} \Vert_{ps}^p
	\end{split}
    \end{align}
    by H\"older's inequality.  Now we choose $ \eps_1$ such that 
    \begin{align*}
	\eps_1\l(M_{16}+ M_{14} G(\kappa, u) \vert \rand \vert^{1/{t'}}\r)= \frac{\kappa p+1}{4(\kappa+1)^p}.
    \end{align*}
    Applying \eqref{bound} to \eqref{dis2} and summarizing the constants results in
    \begin{align}\label{dis4}
	\Vert u u_h^{\kappa} \Vert_{1,p}^p \le M_{20}(\kappa, u,v) [\Vert u u_h^{\kappa} \Vert_{ps}^p+1]
    \end{align}
    with a constant $ M_{20}(\kappa, u,v) $ depending on $ \kappa$ and on the solution pair $(u,v)$, see the calculations above. 
    
    Now we are in the position to use the Sobolev embedding theorem on the left-hand side of \eqref{dis4}. We have
    \begin{align}\label{dis5}
	\Vert u u_h^{\kappa} \Vert_{p^*} \le c_{\Omega} \Vert u u_h^{\kappa} \Vert_{1,p} \le M_{21}(\kappa, u,v) \l[\Vert u u_h^{\kappa} \Vert_{ps}^p+ 1 \r]^{\frac{1}{p}}.
    \end{align}
    Since, due to \eqref{s}, $ ps< p^*$, we can start with the bootstrap arguments. Choosing $ \kappa_1$ such that $ (\kappa_1+1) ps= p^*$, \eqref{dis5} becomes
    \begin{align}\label{dis6}
	\begin{split}
	    \Vert u u_h^{\kappa_1} \Vert_{p^*} 
	    & \le M_{21}(\kappa_1, u,v) \l[\Vert u u_h^{\kappa_1} \Vert_{ps}^p+ 1 \r]^{\frac{1}{p}}\\
	    & \le M_{21}(\kappa_1, u,v) \l[\Vert u^{\kappa_1+1} \Vert_{ps}^p+ 1 \r]^{\frac{1}{p}} \\
	    &= M_{21}(\kappa_1, u,v) \l[\Vert u \Vert_{p^*}^{(\kappa_1+1)p}+ 1 \r]^{\frac{1}{p}}< \infty,
	\end{split}
    \end{align}
    where we have used the estimate $ u_h(x) \le u(x) $ for a.e.\,$ x \in \Omega$. The usage of Fatou's Lemma as $ h \to \infty $ in \eqref{dis6} gives
    \begin{align}\label{dis7}
	\Vert u \Vert_{(\kappa_1+1)p^*}= \Vert u^{\kappa_1+1} \Vert_{p^*}^{\frac{1}{\kappa_1+1}} \le M_{22}(\kappa_1, u,v) \l[\Vert u \Vert_{p^*}^{(\kappa_1+1)p}+1 \r]^{\frac{1}{(\kappa_1+1)p}}< \infty.
    \end{align}
    Hence, $ u \in L^{(\kappa_1+1)p^*}(\Omega)$. Repeating the steps from \eqref{dis5}-\eqref{dis7} for each $ \kappa$, we choose a sequence with the following properties
    \begin{align*}
	\begin{split}
	    & \kappa_2: (\kappa_2+1)ps = (\kappa_1+1) p^*,\\
	    & \kappa_3: (\kappa_3+1)ps = (\kappa_2+1) p^*,\\
	    & \qquad \vdots \qquad \qquad \qquad \qquad \vdots \quad \, .
	\end{split}
    \end{align*}
    Observe that the sequence $(\kappa_n)$ is constructed in such a way that $ \kappa_n+ 1= (\frac{p^*}{ps})^n$ for every $n \in \N$, with $\frac{p^*}{ps}> 1$, taking into account \eqref{s}.
    This implies that
    \begin{align}\label{dis8}
	\Vert u \Vert_{(\kappa+1)p^*} \le M_{23}(\kappa, u,v)
    \end{align}
    for any finite  $ \kappa>0$ with $ M_{23}(\kappa, u,v) $ being a positive constant which depends both on $ \kappa $ and on the solution pair $(u,v)$ itself. Therefore, $ u \in L^r(\Omega)$ for any $ r<\infty$. 

    Now we are going to prove that $ u \in L^r(\rand)$ for any finite $r$. To this end, let us consider again inequality \eqref{dis2}, that is,
    \begin{align}\label{dis9}
	\begin{split}
	    & \frac{\kappa p+1}{2(\kappa+1)^p} \Vert u u_h^{\kappa} \Vert_{1,p}^p\\ 
	    & \le M_{19} \l(\frac{\kappa p+1}{(\kappa+1)^p}+ 1+ (\kappa p+1) L(\kappa, u) \vert \Omega \vert^{1/{s'}} \r) \Vert u u_h^{\kappa} \Vert_{ps}^p \\
	    &\qquad + (M_{16}+ M_{14} G(\kappa, u) \vert \rand \vert^{1/{t'}}) \Vert u u_h^{\kappa} \Vert_{p t, \rand}^p+ M_{17} (\kappa+1).
	\end{split}
    \end{align}
    Exploiting \eqref{dis8}, inequality \eqref{dis9} can be written in the simple form
    \begin{align}\label{dis10}
	\Vert u u_h^{\kappa} \Vert_{1,p} \le M_{24}(\kappa, u,v) \l[\Vert u u_h^{\kappa} \Vert_{p t, \rand}^p+1 \r]^{\frac{1}{p}}.
    \end{align}
    Applying the embedding $\gamma\colon\Wp{p}\to \Lprand{p_*}$ to the right-hand side of \eqref{dis10} gives
    \begin{align*}%\label{dis11}
	\Vert u u_h^{\kappa} \Vert_{p_*, \rand} \le c_{\rand} \Vert u u_h^{\kappa} \Vert_{1,p} \le M_{25}(\kappa, u,v) \l[\Vert u u_h^{\kappa} \Vert_{p t, \rand}^p+ 1 \r]^{\frac{1}{p}}.
    \end{align*}
    Since $ p t< p_*$, we can proceed as before with a bootstrap argument, thus obtaining
    \begin{align*}%\label{dis14}
	\Vert u \Vert_{(\kappa+1)p_*, \rand} \le M_{26}(\kappa, u,v)
    \end{align*}
    for any finite number $ \kappa$ with $ M_{26}(\kappa, u,v)$ being a positive constant depending on $\kappa$ and on the solution pair $(u,v)$. Hence, $u \in L^r(\rand)$ for every $ r<\infty$. Combining this with the first part of the proof shows that $ u \in L^r(\close)$ for every finite $ r$. The same arguments can be applied for the function $v$ starting with the second equation in \eqref{weak}. This completes the proof.
\end{proof}

The next result states the $ L^{\infty}$-boundedness of weak solutions of problem \eqref{problem}.

\begin{theorem}\label{th2}
    Let $ \Omega \subset \RN, N>1$, be a bounded domain with a Lipschitz boundary $\rand$ and let the hypotheses (H) be satisfied. Then, for any weak solution $ (u,v) \in W^{1,p}(\Omega) \times W^{1,q}(\Omega)$ it holds $ (u, v) \in L^{\infty}(\close) \times L^{\infty}(\close)$.
\end{theorem}

\begin{proof}
    Let $ (u,v) \in W^{1,p}(\Omega) \times W^{1,q}(\Omega)$ be a weak solution of problem \eqref{problem}. As in the proof of Theorem \ref{th1} we will suppose  that $ u, v \ge 0$ and we only prove that $ u \in L^{\infty}(\close)$, since the proof that $ v \in L^{\infty}(\close)$ works in a similar way. We repeat the proof of Theorem \ref{th1} until inequality \eqref{dis}, that is
    \begin{align}\label{dis15}
	\begin{split}
	    &\frac{\kappa p+1}{(\kappa+1)^p} \Vert u u_h^{\kappa} \Vert_{1,p}^p\\
	    & \le M_{27} \l(\frac{\kappa p+1}{(\kappa+1)^p}+ 1 \r) \Vert u u_h^{\kappa} \Vert_{ps}^p+ M_{28}(\kappa p+1) \into u^{p^*} u_h^{\kappa p} \,dx \\
	    & \qquad + M_{29} \int_{\rand} u^{p_*} u_h^{\kappa p} \,d\sigma+ M_{30} \Vert u u_h^{\kappa} \Vert_{p t, \rand}^p+ M_{31}(\kappa+1).
	\end{split}
    \end{align}
    Recall that $ ps< p^*$ and $ pt< p_*$. Hence, we can fix numbers $ p_1 \in (ps, p^*)$ and $ p_2 \in (p t, p_*)$. Then, by H\"older's inequality and the $L^r(\close)$-boundedness of $u$ for any finite $r$, see Theorem \ref{th1}, we have for the terms on the right-hand side of \eqref{dis15} the following
    \begin{align}\label{right}
	\begin{split}
	    \Vert u u_h^{\kappa} \Vert_{ps}^p 
	    & \le \vert \Omega \vert^{\frac{p_1- ps}{ p_1s}} \l( \into (u u_h^{\kappa})^{p_1} \,dx \r)^{\frac{p}{p_1}} 
	    \le M_{32} \Vert u u_h^{\kappa} \Vert_{p_1}^p, \\
	    \into u^{p^*} u_h^{\kappa p} \,dx &= \into u^{p^*-p} (u u_h^{\kappa})^p \,dx \\
	    & \le \l(\into u^{\frac{p^*-p}{p_1-p} p_1} \,dx \r)^{\frac{p_1-p}{p_1}} \l(\into (u u_h^{\kappa})^{p_1} \,dx \r)^{\frac{p}{p_1}}\\ 
	    & \le M_{33} \Vert u u_h^{\kappa} \Vert_{p_1}^p, \\
	    \int_{\rand} u^{p_*} u_h^{\kappa p} \,d\sigma &= \int_{\rand} u^{p_*-p} (u u_h^{\kappa})^p \,d\sigma \\
	    & \le \l(\int_{\rand} u^{\frac{p_*-p}{p_2-p} p_2} \,d\sigma \r)^{\frac{p_2-p}{p_2}} \l(\int_{\rand} (u u_h^{\kappa})^{p_2} \,d\sigma \r)^{\frac{p}{p_2}}\\ &\le M_{34} \Vert u u_h^{\kappa} \Vert_{p_2, \rand}^p, \\
	    \Vert u u_h^{\kappa} \Vert_{p t, \rand}^p & \le \vert \rand \vert^{\frac{p_2-p t}{p_2 }} \l(\int_{\rand} (u u_h^{\kappa})^{p_2} \,d\sigma \r)^{\frac{p}{p_2}} \le M_{35} \Vert u u_h^{\kappa} \Vert_{p_2, \rand}^p.
	\end{split}
    \end{align}
    Observe that $ M_{33}, M_{34} $ are finite thanks to Theorem \ref{th1}. More precisely, they are such that
    \begin{align*}%\label{norms}
	M_{33}= M_{33} \l(\Vert u \Vert_{\frac{p^*- p}{p_1- p} p_1} \r), \quad M_{34}= M_{34} \l(\Vert u \Vert_{\frac{p_*-p}{p_2- p} p_2, \rand} \r).
    \end{align*}
    Then \eqref{dis15} becomes
    \begin{align}\label{16}
	\begin{split}
	    \frac{\kappa p+ 1}{(\kappa+1)^p} \Vert u u_h^{\kappa} \Vert_{1,p}^p 
	    &\le M_{36} \l(\frac{\kappa p+1}{(\kappa+1)^p}+ \kappa p+ 2\r) \Vert u u_h^{\kappa} \Vert_{p_1}^p\\
	    &\qquad + M_{37} \Vert u u_h^{\kappa} \Vert_{p_2, \rand}^p+ M_{31} (\kappa+1),
	\end{split}
    \end{align}
    where we used the estimates in \eqref{right}. Now we are going to apply again Proposition \ref{proposition_boundary_integral} to the boundary term. This gives, after using H\"older's inequality,
    \begin{align}\label{bdry}
	\begin{split}
	    \Vert u u_h^{\kappa} \Vert_{p_2, \rand}^p 
	    & \le \eps_2 \Vert u u_h^{\kappa} \Vert_{1,p}^p+ \bar c_1 \eps_2^{-\bar c_2} \Vert u u_h^{\kappa} \Vert_p^p \\
	    & \le \eps_2 \Vert u u_h^{\kappa} \Vert_{1,p}^p+ \bar c_1 \eps_2^{-\bar c_2} M_{38} \Vert u u_h^{\kappa} \Vert_{p_1}^p.
	\end{split}
    \end{align}
    Choosing $ \eps_2 $ such that $ M_{37} \eps_2= \frac{\kappa p+1}{2(\kappa+1)^p} $ and applying \eqref{bdry} to \eqref{16} yields
    \begin{align}\label{17}
	\frac{\kappa p+1}{2(\kappa+1)^p} \Vert u u_h^{\kappa} \Vert_{1,p}^p \le \l[M_{39}(\kappa p+ 2)+ M_{40} \bar c_1 \eps_2^{-\bar c_2} \r] \Vert u u_h^{\kappa} \Vert_{p_1}^p+ M_{31} (\kappa+1).
    \end{align}
    Inequality \eqref{17} can be written in the form
    \begin{align*}%\label{18}
	\Vert u u_h^{\kappa} \Vert_{1,p}^p \le M_{41}((\kappa+1)^p)^{M_{42}} \l[\Vert u u_h^{\kappa} \Vert_{p_1}^p+ 1 \r].
    \end{align*}
%     Indeed, taking into account the choice of $ \eps_4$, one simply notes that
%     \begin{align*}
% 	\begin{split}
% 	    & \frac{2(\kappa+1)^p}{\kappa p+1} \l[M_{37}(\kappa p+2)+ M_{38} \bar c_1 \eps_4^{-\bar c_2} \r] \\
% 	    &= 2(\kappa+1)^p \l[M_{37}\l(1+ \frac{1}{\kappa p+1} \r)+ M_{38} \bar c_1 \l(\frac{M_{35} 2 (\kappa+ 1)^p}{\kappa p+1} \r)^{\bar c_2} \frac{1}{\kappa p+1} \r] \\
% 	    & \le M_{39}((\kappa+1)^p)^{M_{40}}.
% 	\end{split}
%     \end{align*}
    By the Sobolev embedding and the $\Lp{r}$-boundedness of $u$ we obtain
    \begin{align}\label{19}
	\begin{split}
	    \Vert u u_h^{\kappa} \Vert_{p^*} 
	    & \le c_{\Omega} \Vert u u_h^{\kappa} \Vert_{1,p} \le M_{43} (\kappa+1)^{M_{44}} \l[\Vert u u_h^{\kappa} \Vert_{p_1}^p+ 1 \r]^{\frac{1}{p}} \\
	    & \le M_{43} (\kappa+1)^{M_{44}} \l[\Vert u^{\kappa+1} \Vert_{p_1}^p+ 1 \r]^{\frac{1}{p}}< \infty.
	\end{split}
    \end{align}
    Applying Fatou's Lemma to \eqref{19} then gives
    \begin{align}\label{20}
	\Vert u \Vert_{(\kappa+1)p^*}= \Vert u^{\kappa+ 1} \Vert_{p^*}^{\frac{1}{\kappa+1}} \le M_{43}^{\frac{1}{\kappa+1}}((\kappa+1)^{M_{44}})^{\frac{1}{\kappa+1}} \l[\Vert u^{\kappa+1} \Vert_{p_1}^p+ 1 \r]^{\frac{1}{(\kappa+1)p}}.
    \end{align}
    Since
    \begin{align*}
	((\kappa+1)^{M_{44}})^{\frac{1}{\sqrt{\kappa+1}}} \ge 1 \quad \text{and} \quad \lim_{\kappa \to \infty} ((\kappa+1)^{M_{44}})^{\frac{1}{\sqrt{\kappa+1}}}= 1,
    \end{align*}
    there exists $ M_{45}> 1 $ such that
    \begin{align}\label{21}
	((\kappa+1)^{M_{44}})^{\frac{1}{\kappa+1}} \le M_{45}^{\frac{1}{\sqrt{\kappa+1}}}.
    \end{align}
    From \eqref{20}, taking \eqref{21} into account, we have
    \begin{align}\label{22}
    \begin{split}
	\Vert u \Vert_{(\kappa+1) p^*}
	& \le M_{43}^{\frac{1}{\kappa+1}} M_{45}^{\frac{1}{\sqrt{\kappa+1}}} \l[\Vert u^{\kappa+1} \Vert_{p_1}^p+ 1 \r]^{\frac{1}{(\kappa+1)p}}.
	\end{split}
    \end{align}  
    Suppose now there exists a sequence $ \kappa_n \to \infty$ such that
    \begin{align*}
	\Vert u^{\kappa_n+1} \Vert_{p_1}^p \le 1,
    \end{align*}
    that is
    \begin{align*}
	\Vert u \Vert_{(\kappa_n+1) p_1} \le 1.
    \end{align*}
    Then, Proposition \ref{proposition_final_interation} implies that $ \Vert u \Vert_{\infty}< \infty$. 
    On the contrary, suppose that there exists $ \kappa_0> 0 $ such that
    \begin{align*}%\label{23}
	\Vert u^{\kappa+1} \Vert_{p_1}^p> 1 
	\quad \text{for every } \kappa \ge \kappa_0.
    \end{align*}
    Then, \eqref{22} becomes
    \begin{align*}%\label{24}
	\Vert u \Vert_{(\kappa+1) p^*} \le M_{43}^{\frac{1}{\kappa+1}} M_{45}^{\frac{1}{\sqrt{\kappa+1}}} \l[2 \Vert u^{\kappa+1} \Vert_{p_1}^p \r]^{\frac{1}{(\kappa+1)p}} \le M_{46}^{\frac{1}{\kappa+1}} M_{45}^{\frac{1}{\sqrt{\kappa+1}}}\Vert u \Vert_{(\kappa+1) p_1}
    \end{align*}
    for every $ \kappa \ge \kappa_0$. 
    
    Now we choose $\kappa$ in the following way
    \begin{align*}%\label{25}
	\begin{split}
	    & \kappa_1: (\kappa_1+1) p_1= (\kappa_0+1) p^*, \\
	    & \kappa_2: (\kappa_2+1)p_1 = (\kappa_1+1) p^*,\\
	    & \kappa_3: (\kappa_3+1)p_1 = (\kappa_2+1) p^*,\\
	    & \qquad \vdots \qquad \qquad \qquad \qquad \vdots \quad \, .
	\end{split}
    \end{align*}
    This leads to
    \begin{align*}
	\Vert u \Vert_{(\kappa_n+1) p^*} \le M_{46}^{\frac{1}{\kappa_n+1}}M_{45}^{\frac{1}{\sqrt{\kappa_n+1}}} \Vert u \Vert_{(\kappa_{n-1}+1)p^*}
    \end{align*}
    for every $ n \in \N$ with  $ (\kappa_n)$ given by $ (\kappa_n+1)= (\kappa_0+1) \l(\frac{p^*}{p_1}\r)^n$. It follows
    \begin{align*}
	\Vert u \Vert_{(\kappa_n+1)p^*} \le M_{46}^{\sum\limits_{i=1}^n \frac{1}{\kappa_i+1}}  M_{45}^{\sum\limits_{i=1}^n \frac{1}{\sqrt{\kappa_i+1}}} \Vert u \Vert_{(\kappa_0+1)p^*}.
    \end{align*}
Since
    \begin{align*}
	\frac{1}{\kappa_i+1}= \frac{1}{\kappa_0+1} \l(\frac{p_1}{p^*}\r)^i \quad \text{and} \quad  \frac{p_1}{p^*}< 1,
    \end{align*}
    there exists $M_{47}> 0 $ such that
    \begin{align*}
	\Vert u \Vert_{(\kappa_n+1)p^*} \le M_{47} \Vert u \Vert_{(\kappa_0+1) p^*}< \infty,
    \end{align*}
    where the right-hand side is finite thanks to Theorem \ref{th1}. Now we may apply again Proposition \ref{proposition_final_interation}. This ensures that $ u \in L^{\infty}(\Omega)$. Moreover, Proposition \ref{proposition_boundedness_boundary} gives $ u \in L^{\infty}(\rand)$ and so, $ u \in L^{\infty}(\close)$. 
\end{proof}

\begin{remark}\label{remark_well_definitedness}
    The conditions on the exponents in hypotheses (H) are not the natural ones. Precisely, in order to have a well-defined weak solution it is enough to require the following assumptions
    \begin{align*}\hspace*{-0.1cm}
     \begin{array}[]{lll}
		        \text{(E1)} \quad  r_1 \le p^*
                & \text{(E2)} \quad  r_2 \le q^*
                &\\[3ex]
		\text{(E3)} \quad b_1 \le p^*-1
                & \text{(E4')} \quad \displaystyle b_2 \leq \frac{q^*}{p^*}(p^*-1)
                & \text{(E5')} \quad  \displaystyle\frac{b_3}{p^*}+ \frac{b_4}{q^*}\leq \frac{p^*-1}{p^*}\\[3ex]
                \text{(E6)} \quad b_5 \le p-1
                & \text{(E7')} \quad \displaystyle b_6 \leq \frac{q}{p^*}(p^*-1)
                & \text{(E8')} \quad\displaystyle \frac{b_7}{p}+ \frac{b_8}{q}\leq \frac{p^*-1}{p^*}\\[3ex]
                \text{(E9')}\displaystyle \quad \tilde b_1 \leq \frac{p^*}{q^*}(q^*-1)
                & \text{(E10)} \quad \displaystyle \tilde b_2 \le q^*-1
                & \text{(E11')} \quad\displaystyle\frac{\tilde b_3}{p^*}+ \frac{\tilde b_4}{q^*}\leq \frac{q^*-1}{q^*}\\[3ex]
                \text{(E12')} \displaystyle\quad  \tilde b_5 \leq \frac{p}{q^*}(q^*-1)
                & \text{(E13)} \quad \displaystyle\tilde b_6 \le q- 1
                & \text{(E14')} \quad\displaystyle\frac{\tilde b_7}{p}+ \frac{\tilde b_8}{q}\leq \frac{q^*-1}{q^*}\\[3ex]
                 \text{(E15)} \quad c_1  \le p_*-1
                & \text{(E16')} \quad \displaystyle c_2 \leq \frac{q_*}{p_*}(p_*-1)
                & \text{(E17')} \quad\displaystyle\frac{c_3}{p_*}+ \frac{c_4}{q_*}\leq \frac{p_*-1}{p_*}\\[3ex]
                \text{(E18')}\displaystyle \quad  \tilde c_1 \leq \frac{p_*}{q_*}(q_*-1)
                & \text{(E19)} \quad \displaystyle\tilde c_6 \le q_*-1
                & \text{(E20')} \quad\displaystyle\frac{\tilde c_3}{p_*}+ \frac{\tilde c_4}{q_*}\leq \frac{q_*-1}{q_*}.
        \end{array}
    \end{align*}
    In order to apply Moser's iteration we needed to strengthen the hypotheses for (E4'), (E5'), (E7'), (E8'), (E9'), (E11'), (E12'), (E14'), (E16'), (E17'), (E18'), and (E20'). We also point out that hypotheses (H1) and (H2) are not explicity needed in the proofs of Theorems \ref{th1} and \ref{th2}, but they are necessary to have a well-defined  weak solution as defined in \eqref{weak}.
    
    Furthermore, the bounds obtained in Theorem \ref{th1} and \ref{th2} depend on the data in hypotheses (H) and also on the solution pair $(u,v)$. In particular, the bound for $u$ also depends on $v$ and vice-versa.
\end{remark}

In the last part we want to mention that the results obtained in Theorems \ref{th1} and \ref{th2} can be easily applied to problems of the form \eqref{problem} with a homogeneous Dirichlet condition. Indeed, consider the problem
\begin{equation}\label{problem3}
    \begin{aligned}
	-\divergenz \mathcal A_1(x, u, \nabla u)
	&= \mathcal B_1(x, u, v, \nabla u, \nabla v) \qquad 
	&& \text{in } \Omega, \\
	-\divergenz \mathcal A_2(x, v, \nabla v)
	&= \mathcal B_2(x, u, v, \nabla u, \nabla v) \qquad 
	&& \text{in } \Omega, \\
	u=v
	&= 0 
	&& \text{on } \rand.
    \end{aligned}
\end{equation}
We suppose the following assumptions on the data in problem \eqref{problem3}.

\begin{enumerate}[leftmargin=0.8cm]
 \item[($\tilde{\text{H}}$)] 
    The functions $ \mathcal A_i\colon \Omega \times \R   \times \R^N  \to \R^N$ and $\mathcal B_i\colon \Omega \times \R \times \R \times \R^N \times \R^N \to \R$, $i= 1,2 $, are Carath\'eodory functions such that
    \begin{align*}
        \text{($\tilde{\text{H}}$1}) \quad  
	    & \vert \mathcal A_1(x, s, \xi) \vert \le A_1 \vert \xi \vert^{p-1}+  A_2 \vert s \vert^{r_1 \frac{p-1}{p}}+ A_3,  \\
        \text{($\tilde{\text{H}}$2}) \quad 
	    & \vert \mathcal A_2(x, t, \zeta) \vert \le  \tilde A_1 \vert \zeta \vert^{q-1}+ \tilde A_2  \vert t \vert^{r_2 \frac{q-1}{q}}+ \tilde A_3,  \\
        \text{($\tilde{\text{H}}$3}) \quad 
	    & \mathcal A_1(x, s, \xi) \cdot \xi \ge A_4 \vert \xi \vert^p- A_5 \vert s \vert^{r_1}- A_6,   \\
        \text{($\tilde{\text{H}}$4}) \quad 
	    & \mathcal A_2(x, t, \zeta) \cdot \zeta \ge  \tilde A_4 \vert \zeta \vert^q-  \tilde A_5 \vert t \vert^{r_2}- \tilde A_6,  \\
        \text{($\tilde{\text{H}}$5}) \quad 
	    & \vert \mathcal B_1(x, s, t, \xi, \zeta) \vert \le B_1 \vert s \vert^{b_1}+ B_2 \vert t \vert^{b_2}+ B_3 \vert s \vert^{b_3} \vert t \vert^{b_4}+ B_4 \vert \xi \vert^{b_5}\\
	    & \hspace*{3cm}+ B_5 \vert \zeta \vert^{b_6}+ B_6 \vert \xi \vert^{b_7} \vert \zeta \vert^{b_8}+ B_7,  \\
        \text{($\tilde{\text{H}}$6}) \quad 
	    & \vert \mathcal B_2(x, s, t, \xi, \zeta) \vert \le  \tilde B_1 \vert s \vert^{\tilde b_1}+ \tilde B_2 \vert t \vert^{\tilde b_2}+ \tilde B_3  \vert s \vert^{\tilde b_3} \vert t \vert^{\tilde b_4}+ \tilde B_4  \vert \xi \vert^{\tilde b_5}\\
	    & \hspace*{3cm}
	    + \tilde B_5  \vert \zeta \vert^{\tilde b_6}+ \tilde B_6  \vert \xi \vert^{\tilde b_7} \vert \zeta \vert^{\tilde b_8}+ \tilde B_7, 
\end{align*}
for a.e.\,$x \in \Omega$, for all $ s, t \in \R $, and for all $ \xi, \zeta \in \R^N $, with nonnegative constants $ A_i, \tilde A_i, B_j, \tilde B_j\, (i \in \{1, \dots, 6\}, j \in \{1, \dots, 7\}) $ and with $ 1< p, q< \infty$. Moreover, the exponents $ b_i, \tilde b_i,  r_1, r_2  \, (i \in \{1, \dots, 8\} )$ are nonnegative and satisfy the following assumptions
\begin{align*}
      \hspace*{0.5cm}\begin{array}[]{lll}
		        \text{(E1)} \quad  r_1 \le p^*
                & \text{(E2)} \quad  r_2 \le q^*
                &\\[3ex]
		\text{(E3)} \quad b_1 \le p^*-1
                & \text{(E4)} \quad \displaystyle b_2 < \frac{q^*}{p^*}(p^*-p)
                & \text{(E5)} \quad  \displaystyle\frac{b_3}{p^*}+ \frac{b_4}{q^*}< \frac{p^*-p}{p^*}\\[3ex]
                \text{(E6)} \quad b_5 \le p-1
                & \text{(E7)} \quad \displaystyle b_6 < \frac{q}{p^*}(p^*-p)
                & \text{(E8)} \quad\displaystyle \frac{b_7}{p}+ \frac{b_8}{q}< \frac{p^*-p}{p^*}\\[3ex]
                \text{(E9)}\displaystyle \quad \tilde b_1 < \frac{p^*}{q^*}(q^*-q)
                & \text{(E10)} \quad \displaystyle \tilde b_2 \le q^*-1
                & \text{(E11)} \quad\displaystyle\frac{\tilde b_3}{p^*}+ \frac{\tilde b_4}{q^*}< \frac{q^*-q}{q^*}\\[3ex]
                \text{(E12)} \displaystyle\quad  \tilde b_5 < \frac{p}{q^*}(q^*-q)
                & \text{(E13)} \quad \displaystyle\tilde b_6 \le q- 1
                & \text{(E14)} \quad\displaystyle\frac{\tilde b_7}{p}+ \frac{\tilde b_8}{q}< \frac{q^*-q}{q^*},
        \end{array}
\end{align*}
where the numbers $ p^*, p_*, q^*, q_*$ are defined by \eqref{critical1} and \eqref{critical}. 
\end{enumerate}

A couple $ (u,v) \in W^{1,p}_0(\Omega) \times W^{1,q}_0(\Omega) $ is said to be a weak solution of problem \eqref{problem3} if
        \begin{align*}%\label{weak2}
            \begin{split}
                    \into \mathcal A_1(x, u, \nabla u) \cdot \nabla \varphi \,dx&= \into \mathcal B_1(x, u, v, \nabla u, \nabla v) \varphi \,dx\\
                    \into \mathcal A_2(x, v, \nabla v) \cdot \nabla \psi \,dx&= \into \mathcal B_2(x, u, v, \nabla u, \nabla v) \psi \,dx
            \end{split}
        \end{align*}
holds for all $ (\varphi, \psi) \in W^{1,p}_0(\Omega) \times W^{1,q}_0(\Omega) $.

We can state the following result for problem \eqref{problem3}.

\begin{theorem} \label{th3}
    Let $ \Omega \subset \R^N, N> 1$, be a bounded domain with Lipschitz boundary $ \rand $ and let hypotheses $(\tilde{H})$ be satisfied. Then, every weak solution $ (u, v) \in W^{1,p}_0(\Omega) \times W^{1,q}_0(\Omega) $ of problem \eqref{problem3} belongs to $ L^{\infty}(\close) \times L^{\infty}(\close)$.
\end{theorem}

The proof of Theorem \ref{th3} works exactly in the same way as the proofs of Theorems \ref{th1} and \ref{th2}.

\section*{Acknowledgment}

The authors wish to thank the two anonymous referees for their corrections and useful remarks.

The first author thanks the University of Technology Berlin (Technische Universit\"{a}t Berlin) for the kind hospitality during a research stay in October 2018 and the second author thanks the University of Catania (which is the former university of the first author) for the kind hospitality during a research stay in March 2019.

\end{document}